\documentclass[10pt]{article}
\usepackage{amsmath,amssymb,amsthm}
\usepackage{graphicx,subfig,float,url,color}
\usepackage{mathrsfs}
\usepackage[colorlinks=true]{hyperref}
\usepackage{pdfsync}
\usepackage{stmaryrd} % \llbracket et \rrbracket
\usepackage{dsfont}

\graphicspath{{fig/}}

\newcommand{\R} {\mathbb{R}}

\renewcommand{\geq}{\geqslant}
\renewcommand{\leq}{\leqslant}

\newtheorem{theorem}{Theorem}

\theoremstyle{definition}
\newtheorem{remark}{Remark}

\title{Combination of direct methods and homotopy in numerical optimal control: application to the optimization of chemotherapy in cancer}

\author{Antoine Olivier\thanks{Sorbonne Universit\'es, UPMC Univ Paris 06, CNRS UMR 7598, Laboratoire Jacques-Louis Lions, F-75005, Paris, France}  \footnotemark[3]\; \, Camille Pouchol\footnotemark[1] \thanks{INRIA Team Mamba, INRIA Paris, 2 rue Simone Iff, CS 42112, 75589 Paris, France} \thanks{e-mail: \href{mailto:olivier@ljll.math.upmc.fr}{olivier@ljll.math.upmc.fr} (corresponding author), \href{mailto:pouchol@ljll.math.upmc.fr}{pouchol@ljll.math.upmc.fr}}}

\date{}

\begin{document}

\maketitle

\begin{abstract}
We consider a state-constrained optimal control problem of a system of two non-local partial-differential equations, which is an extension of the one introduced in a previous work in mathematical oncology. The aim is to minimize the tumor size through chemotherapy while avoiding the emergence of resistance to the drugs. The numerical approach to solve the problem was the combination of direct methods and continuation on discretization parameters, which happen to be insufficient for the more complicated model, where diffusion is added to account for mutations. 
In the present paper, we propose an approach relying on changing the problem so that it can theoretically be solved thanks to a Pontryagin Maximum Principle in infinite dimension. This provides an excellent starting point for a much more reliable and efficient algorithm combining direct methods and continuations. The global idea is new and can be thought of as an alternative to other numerical optimal control techniques.

\end{abstract}

%\keywords{Optimal control, Direct methods, Homotopy, Mathematical oncology, Pontryagin maximum principle.}

%\subclass{49M05, 49M25, 92C50.}

\section{Introduction}

The motivation for this work is the article~\cite{Pouchol2016}, itself initiated by~\cite{Lorz2013}. In the former, the subject was the theoretical and numerical analysis of an optimal control problem coming from oncology. Through chemotherapy, it consists of minimizing the number of cancer cells at the end of a given therapeutic window. The underlying model was an integro-differential system for the time-evolution of densities of cancer and healthy cells, structured by their continuous level of resistance to chemotherapeutic drugs.
The model took into account cell proliferation and death, competition between the cells, and the effect of chemotherapy on them. The optimal control problem also incorporated constraints on the doses of the drugs, as well as constraints on the tumor size and on the healthy tissue. \par

In~\cite{Pouchol2016}, the numerical resolution of the optimal control problem was made through a direct method, thanks to a discretization both in time and in the phenotypic variable. It led to a complex nonlinear constrained optimization problem, for which even efficient algorithms will fail for large discretization parameters because they require a good initial guess. To overcome this, the idea was to perform (with AMPL and IPOPT, see below) a continuation on the discretization parameters, starting from low values (\textit{i.e.}, a coarse discretization) for which the optimization algorithm converges regardless of the starting point. 

A clear optimal strategy emerged from these numerical simulations when the final time was increased. It roughly consists of first using as few drugs as possible during a long first phase to avoid the emergence of resistance. Cancer cells would hence concentrate on a sensitive phenotype, allowing for an efficient short second phase with the maximum tolerated doses. 

The model of~\cite{Pouchol2016} did not include \textit{epimutations}, namely heritable changes in DNA expression which are passed from one generation of cells to the others, which are believed to be very frequent in the life-time of a tumor. Our aim here is to numerically address the optimal control problem with the epimutations modeled through diffusion operators (Laplacians), in order to test the robustness of the optimal strategy. 
%The equations have thus become partial differential equations, which are non-local because of the competition term.

However, the previous numerical technique already failed (even without Laplacians) to get fine discretizations when the final time is very large: the optimization stops converging when the discretization parameters are large. The values reached for the discretization in time were enough to observe the optimal structure, in particular all the arcs that were expected for theoretical reasons.

The addition of Laplacians significantly increases the run-time and again fails to work once the discretization parameters are too large when the final time itself is large, and some arcs become difficult to observe. We thus have to find an alternative method to see whether the optimal strategy found in~\cite{Pouchol2016} is robust with respect to adding the effect of epimutations.

This article is devoted to the presentation of a method which, up to our knowledge, is new. In our case, it provides a significant improvement in run-time and precision, and shows that the optimal strategy keeps an analogous structure when epimutations are considered.  The method relies on the two following steps:
\begin{itemize}
\item first, simplify the optimal control problem up to a point where we can show that, thanks to a Pontryagin Maximum Principle (PMP) in infinite dimension, the optimal controls are bang-bang and thus can be reduced to their switching times, which are very easy to estimate numerically. This is equivalent to setting several coefficients to $0$ in the model.
\item second, perform a continuation on these parameters on the optimization problems obtained with a direct method, starting from the simplified problem all the way back to the full optimal control problem.
\end{itemize}
It allows us to start the homotopy method on this simplified optimization problem with an already fine discretization, actually much finer than the maximal values which could be obtained with the previous homotopy method. 
We also believe that the theoretical result obtained for the simplified optimal control problem can serve as the starting step for many other optimal control problems of related models in mathematical biology.

\paragraph{Numerical optimal control and novelty of the approach.}
Discretizing the time variable, control and state variables to approximate a control problem for an ODE (which is an optimization problem in infinite dimension) by a finite-dimensional optimization problem has now become the most standard way of proceeding. These so-called direct methods thus lead to using efficient optimization algorithms, for example through the combination of automatic differentiation softwares (such as the modeling language AMPL, see~\cite{Fourer2002}) and expert optimization routines (such as the open-source package IPOPT, see~\cite{Waechter2006}). 

Another approach is to use indirect methods, where the whole process relies on a PMP, leading to a shooting problem on the adjoint vector. Numerically, one thus needs to find the zeros of an appropriate function, which is usually done through a Newton-like algorithm. 
For a comparison of the advantages and drawbacks of direct and indirect methods, we refer to the survey~\cite{Trelat2012}.

For both direct and indirect methods, the numerical problem shares at least the difficulty of finding an initial guess leading to convergence of the optimization algorithm or the Newton algorithm, respectively (it is well known that Newton algorithms can have a very small domain of convergence). To tackle this issue in the case of indirect methods, it is very standard to use homotopy techniques, for instance to simplify the problem so that one can have a good idea for a starting point as in~\cite{Cerf2012,Chupin2016}, or to change the cost in order to benefit from convexity properties, as in~\cite{Haberkorn2006,Caillau2012}.
Besides, when studying optimal control problem for ODE systems, a common approach is to use of so-called hybrid methods, in order to take advantage from the better convergence properties of the direct method and the high accuracy provided by the indirect method. We refer to~\cite{Trelat2012,Bulirsch1993,Pesch1994,vonStryk1992} for further developments on this subject. 

We have found the combination of direct methods and continuation (such as the one done in~\cite{Pouchol2016}) to be much less common in the literature, see however~\cite{Bulirsch1993}. For a mathematical investigation of why continuation methods are mathematically valid, see~\cite{Trelat2012}. 

It is however believed that direct methods typically lead to optimization problems with several local minima~\cite{Trelat2012}, as it could happen for the starting problem (with low discretization), which has yet no biological meaning. This implies one important drawback of a continuation on discretization parameters with direct methods: although the algorithm will quickly converge in such cases, one cannot \textit{a priori} exclude that one will get trapped in local minima that are meaningless, with the possibility for such trapping to propagate through the homotopy procedure.

Our approach of simplifying the optimal control problem so that it can be analyzed with theoretical tools such as a PMP is a way to address the previous problem and to decrease the computation time. The simplified optimal control problem, once approximated by a direct method, will indeed efficiently be solved even with a very refined discretization. Therefore, another original aspect of our work, due to the complex PDE structure of the model, is the use of the PMP in view of building an initial guess for the direct method, in contrast with the hybrid approach we described for ODE systems, where direct methods serve to initialize shooting problems.

More generally, we advocate for the strategy of trying to simplify the problem, testing whether a PMP can provide a good characterization of the optimal controls. Then continuation with direct methods are performed to get back to the original and more difficult one. We believe that this can always be tried as a possible strategy to solve any optimal control problem (ODE or PDE) numerically. 

\paragraph{Outline of the paper.}
The paper is organized as follows. Section \ref{Section2} is devoted to a detailed presentation of the optimal control problem and the results that were obtained in~\cite{Pouchol2016}. Section \ref{Section3} presents the simplified optimal control problem together with the application of a Pontryagin Maximum Principle in infinite dimension which almost completely determines the optimal controls. In Section \ref{Section4}, we thoroughly explain how direct methods for the optimal control of PDEs and continuations can be combined to solve a given PDE optimal control problem. We then combine these techniques and the result of Section \ref{Section3} to build an algorithm solving the complete optimal control problem. In Section \ref{Section5} the numerical simulations obtained thanks to the algorithm are presented. Finally, we will give some perspectives in Section \ref{Section6} before concluding in Section \ref{Section7}.

\section{Modeling Approach and Optimal Control Problem}
\label{Section2}
\subsection{Modeling Approach}
Let us first explain the modeling approach, which is based on the classical logistic ODE \[\frac{d N}{d t} = \left(r - d N \right) N.\]
In this setting, individuals $N(t)$ have a net selection rate $r$, together with an additional death term $d N$ increasing with $N$: the more individuals, the more death due to competition for resources and space. \par

If the individuals have different selection and death rates $r(x)$ and $d(x)$ depending on a continuous variable $x$ which we will call \textit{phenotype} (the size of the individual, for example), then a natural extension to the previous model is to study the density of individuals $n(t,x)$ of phenotype $x$, at time $t$, satisfying the integro-differential equation \[\frac{\partial n}{\partial t}(t,x) = \big(r(x) - d(x) \rho(t) \big) n(t,x),\] where \[\rho(t):=\int n(t,x) \, dx.\] \par

At this stage, individuals do not change phenotype over time, nor can they give birth to offspring with different phenotypes. Accounting for such a possibility consists in modeling random mutations (respectively random epimutations), \textit{i.e.}, heritable changes in the DNA (respectively heritable changes in DNA expression). The model is complemented with a diffusion term and takes the form \[\frac{\partial n}{\partial t}(t,x) = \big(r(x) - d(x) \rho(t) \big) n(t,x) + \beta \Delta n(t,x),\] together with Neumann boundary conditions if $x$ lies in a bounded domain, thus becoming a non-local partial differential equation because of the integral term $\rho$. \par

Such so-called \textit{selection-mutation} models are actively studied as they represent a suitable mathematical framework for investigating how selection occurs in various ecological scenarios~\cite{Diekmann2004,Diekmann2005,Perthame2006}, thus belonging to the branch of mathematical biology called adaptive dynamics. When $\beta = 0$, the previous model indeed leads to asymptotic selection: $n$ converges to a sum of Dirac masses located on the set of phenotypes on which $\frac{r}{d}$ reaches its maximum~\cite{Pouchol2016,Perthame2006}. In particular, if this set is reduced to a singleton $x_0$ it holds that $\frac{n(t, \cdot)}{\rho(t)}$ weakly converges to a Dirac at $x_0$ as $t$ goes to $+\infty$. \par

\subsection{The Optimal Control Problem}

The model considered in this paper is an extension of the one studied in~\cite{Pouchol2016} by the addition of epimutations (it is believed that mutations occur on a too long time-scale and are consequently neglected~\cite{Chisholm2016}). It describes the dynamics of two populations of cells, healthy and cancer cells, which are both structured by a trait $x \in [0,1]$ representing resistance to chemotherapy, which ranges from sensitiveness ($x=0$) to resistance ($x=1$). $x$ is taken to be a continuous variable because resistance to chemotherapy can be correlated to biological characteristics which are continuous, see~\cite{Chisholm2016} for more details. Chemotherapy is modeled by two functions of time $u_1$ and $u_2$, standing for the rate of administration of cytotoxic drugs and cytostatic drugs, respectively. The first type of drug actively kills cancer cells, while the second slows down their proliferation. \par

The system of equations describing the time-evolution of the density of healthy cells $n_H(t,x)$ and cancer cells $n_C(t,x)$ is given by 
\begin{equation*}
\begin{split}%R_C(x,\rho_C,\rho_H,0,u_2)
\frac{\partial n_H}{\partial t}(t,x) &= 
\left[\frac{r_H(x)}{1+\alpha_H u_2(t)} - d_H(x) I_H(t) - u_1(t) \mu_H(x) \right]   n_H(t,x)  + \beta_H \Delta n_H(t,x), \\ \vspace{.8em}
\frac{\partial n_C}{\partial t}(t,x) & = 
\left[\frac{r_C(x)}{1+\alpha_C u_2(t)} - d_C(x) I_C(t) - u_1(t) \mu_C(x)\right]  n_C(t,x) + \beta_C \Delta n_C(t,x),
\end{split}
\end{equation*}
starting from an initial condition $(n_H^0, n_C^0)$ in $C([0,1])^2$, with Neumann boundary conditions in $x=0$ and $x=1$.

Let us describe in more details the different terms and parameters appearing above, with the functions $r_H$, $r_C$, $d_H$, $d_C$, $\mu_H$ $\mu_C$ all continuous and non-negative on $[0,1]$, with $r_H$, $r_C$, $d_H$, $d_C$ positive on $[0,1]$.
\begin{itemize} 
\item
The terms $\frac{r_H(x)}{1+\alpha_H u_2(t)}$, $\frac{r_C(x)}{1+\alpha_C u_2(t)}$ stand for the selection rates lowered by the effect of the cytostatic drugs, with \[\alpha_H < \alpha_C.\]
\item 
The non-local terms $d_H(x) I_H(t)$, $d_C(x) I_C(t)$ are added death rates to the competition inside and between the two populations, with \[I_H := a_{HH} \rho_H + a_{HC} \rho_C, \; I_C := a_{CC} \rho_C + a_{CH} \rho_H\] and as before \[\rho_i(t) = \int_0^1 n_i(t,x) \, dx, \; i=H,C.\] We make the important assumption that the competition inside a given population is greater than between the two populations: \[a_{HC}<a_{HH}, \; a_{CH} < a_{CC}.\]
\item 
The terms $\mu_H(x) u_1(t)$, $\mu_C(x) u_1(t)$ are added death rates due to the cytotoxic drugs. Owing to the meaning of $x=0$ and $x=1$, $\mu_H$ and $\mu_C$ are taken to be decreasing functions of $x$.
\item 
The terms $\beta_H \Delta n_H(t,x)$ and $\beta_C \Delta n_C(t,x)$ model the random epimutations, with their rates $\beta_H$, $\beta_C$ such that 
\[\beta_H<\beta_C,\]
because cancer cells mutate faster than healthy cells.
\end{itemize}

Finally, for a fixed final time $T$ we consider the optimal control problem (denoted in short by $(\textbf{OCP}_{\textbf{1}})$) of minimizing the number of cancer cells at the end of the time-frame
\begin{equation*}
\inf \, \rho_C(T) 
\end{equation*}
as a function of the $L^\infty$ controls $u_1$, $u_2$ subject to $L^\infty$ constraints for the controls and two state constraints on $(\rho_H, \rho_C)$, for all $0 \leq t \leq T$:
\begin{itemize} 
\item The maximum tolerated doses cannot be exceeded: 
\[0 \leq u_1(t) \leq u_1^{max}, \;\; 0 \leq u_2(t) \leq u_2^{max}.\]

\item The tumor cannot be too big compared to the healthy tissue:
\begin{equation} \label{contHC} \frac{\rho_H(t)} {\rho_H(t) + \rho_C(t)} \geq \theta_{HC},
\end{equation} 
with $0 < \theta_{HC} <1$.

\item Toxic side-effects must remain controlled:
\begin{equation} \label{contH} \rho_H(t) \geq \theta_{H} \rho_H(0),\end{equation}
with $0 < \theta_{H} <1$.
\end{itemize}

Optimal control problems applied to cancer therapy have started being considered long ago, see~\cite{Schaettler2015} for a complete presentation. However, the usual way of taking resistance into account is to consider that cells are either resistant or sensitive, leading to ODE models, as for example in~\cite{Costa1992,kimmelSwierniak2006,Ledzewicz2006,Ledzewicz2014,Carrere2017}. Considering both a continuous modeling of resistance and the effect of chemotherapy is more recent, as in~\cite{Pouchol2016,Chisholm2016,Lorz2013a,Greene2014,Lorenzi2015a}. We also mention some cases where an additional space variable is considered~\cite{Lorz2013,Lorz2015}.

\begin{remark}
\label{rq:OtherObjective}
In practice, other objective functions can be deemed pertinent. For example, if there is no hope of actually getting rid of the tumour, the goal might be to try and control its size on the whole interval. Thus, we will also consider objective functions of the form of convex combinations $$ \lambda_0 \int_0^T \rho_C(s) \, ds + (1- \lambda_0) \rho_C(T), $$ 
where $0 \leq \lambda_0 \leq 1$. 
For $\lambda_0 = 0$, we recover the previous objective function, while for $\lambda_1 = 1$ only the $L^1$ norm of $\rho_C$ is considered.
\end{remark}

\subsection{Previous Results}
In~\cite{Pouchol2016}, we studied this system and the optimal control problem both theoretically and numerically in the case of selection exclusively, namely for $\beta_H = \beta_C = 0$.

First, we proved that for constant controls (\textit{i.e.}, constant doses), the generic behavior is the convergence of both densities to Dirac masses. When these doses are high, the model thus reproduces the clinical observation that high doses usually fail at controlling the tumor size on the long run. They might indeed initially lead to a decrease of the overall cancerous population. However, this is the consequence of only the sensitive cells being killed, while the most resistant cells are selected (in our mathematical framework, this corresponds to the cancer cell density concentrating on a resistant phenotype). Further treatment is then inefficient and the tumor starts growing again. \par

As for the optimal control problem which is our focus in this work, the main findings without diffusion were the following: when the final time $T$ becomes large, the optimal controls acquire some clear structure which is made of two main phases. 

\begin{itemize} 
\item
First, there is a long phase with low doses of drugs ($u_1 = 0$ with our parameters), along which the constraint \eqref{contHC} quickly saturates. At the end of this first long arc, both densities have concentrated on a sensitive phenotype.
\item
Then, there is a second short phase, which is the concatenation of two arcs. The first one is a free arc (no state constraint is saturated) along which $u_1 = u_1^{max}$ and $u_2 = u_2^{max}$, with a quick decrease of both cell numbers $\rho_H$ and $\rho_C$, up until the constraint on the healthy cells \eqref{contH} saturates. The last arc is constrained on \eqref{contH} with boundary controls ($u_2 = u_2^{max}$ with our parameters), allowing for a further decrease of $\rho_C$.
\end{itemize}

In other words, the optimal strategy is to let the cell densities concentrate on sensitive phenotypes so that the full power of the drugs can efficiently be used. This strategy is followed as long as the healthy tissue can endure it, and then lower doses are used to keep on lowering $\rho_C$ while still satisfying the toxicity constraint.

\section{Resolution of a Simplified Model}
\label{Section3} 

\subsection{Simplified Model for one Population with no State Constraints}
We here introduce the simpler optimal control problem. Its precise link with the initial optimal control $(\textbf{OCP}_{\textbf{1}})$ will be explained in Section \ref{Section4}. 
It is based on the equation
\begin{equation}
\label{eq_state}
\frac{\partial n_C}{\partial t}(t,x) = \left[ \frac{r_C(x)}{1+\alpha_C u_2(t)} - d_C(x) \rho_C(t) - \mu_C(x) u_1(t) \right] n_C(t,x),
\end{equation}
starting from $n_C^0$, where $\rho_C(t) = \int^1_0{n_C(t,x)\,dx}$. We denote by $(\textbf{OCP}_{\textbf{0}})$ the optimal control problem
\begin{equation*}
%\label{ocp_0}
\min_{(u_1,u_2) \in \mathcal{U}} \rho_C(T)
\end{equation*}
where $\mathcal{U}$ is the space of admissible controls 
\[
\mathcal{U}:=\left\{ (u_1,u_2) \in L^\infty([0,T],\R) \text{ such that } 0 \leq u_1 \leq u_1^{max},~0 \leq u_2 \leq u_2^{max},~ \text{a.e. on } [0,T]  \right\}.
\]

\subsection{A Maximum Principle in Infinite Dimension}

\paragraph{General statement.}
Let $T$ be a fixed final time, $X$ be a Banach space and $n_0 \in X$, $U$ be a separable metric space. We also consider two mappings $f : [0,T] \times X \times U \rightarrow X$ and $f^0 : [0,T] \times X \times U \rightarrow \R$.

We consider the optimal control problem of minimizing an integral cost, with a free final state $n(T)$:
\begin{equation*}
\inf_{u \in \mathcal{U}} J(u(\cdot)) := \int^T_0{f^0(t,n(t),u(t))\,dt},
\end{equation*}
where $y(\cdot)$ is the solution\footnote{Note that the evolution equation has to be understood in the mild sense 
\[
n(t) = n_0 + \int^t_0{f(s,n(s),u(s))\,ds.}
\]} of 
\begin{equation*}
\dot{n}(t) = f(t,n(t),u(t)), ~~n(0) = n_0.
\end{equation*}

In~\cite[Chapter~4]{li2012}, necessary conditions for optimality are presented, for such problems (they are actually presented in~\cite{li2012} in a more general setting, but for the sake of simplicity, we restrict ourselves to the material required to solve $(\textbf{OCP}_{\textbf{0}})$). The set of these conditions is referred to as a Pontryagin Maximum Principle (PMP).

Under appropriate regularity assumptions on $f$ and $f^0$, it states that any optimal pair $(\overline{n}(\cdot),\overline{u}(\cdot))$ must be such that there exists a nontrivial pair $(p^0, p(\cdot)) \in \R \times C([0,T],X)$ satisfying
\begin{equation}
\label{p0_neg}
p^0 \leq 0,
\end{equation}
%%%%%%%%%%
\begin{equation}
\label{eq_costate}
\dot{p}(t) = - \frac{\partial H}{\partial n}(t,\overline{n}(t),\overline{u}(t), p^0, p(t)),
\end{equation}

\begin{equation}
\label{eq_maximisation}
H(t,\overline{n}(t),\overline{u}(t), p^0, p(t)) = \max_{v \in U} H(t,\overline{n}(t),v, p^0, p(t)),
\end{equation}
where the hamiltonian $H$ is defined as $H(t,n,u,p,p^0) := p^0 f^0(t,n,u) + \langle p,f(t,n,u) \rangle$.

\begin{remark}
\label{rq_freefinalstate}
If the final state is free, \eqref{p0_neg} can be improved to $p_0 <0$\footnote{An extremal in the PMP is said to be normal (resp. abnormal) whenever $p^0 \neq 0$ (resp. $p^0 = 0$). Here, it means that there is no abnormal extremal.} and we have the additional transversality condition:
\begin{equation}
\label{eq_transverse}
p(T) = 0.
\end{equation}
Besides, if the final state were fixed, there would be additional assumptions to check in order to apply the PMP, assumptions that are automatically fulfilled whenever $n(T)$ is free. We refer to~\cite[Chapter 4 - Section 5]{li2012} for more details on this issue.
\end{remark}

\paragraph{Application to the problem  \emph{$(\textbf{OCP}_{\textbf{0}})$}.} By applying the PMP, we derive the following theorem on the optimal control structure.

\begin{theorem}
\label{th_simplified}
Let $({n}_C(\cdot),{u}(\cdot))$ be an optimal solution for $(\textbf{OCP}_{\textbf{0}})$. There exists $t_1 \in [0,T[$ and $t_2 \in [0,T[$ such that
\begin{equation*}
\begin{array}{lr}
u_1(t) = u_1^{max} \mathds{1}_{[t_1,T]}, &~~u_2(t) = u_2^{max} \mathds{1}_{[t_2,T]}. 
\end{array}
\end{equation*}
\end{theorem}

\begin{proof}
Let us define $U := \left\{u=(u_1,u_2) \text{ such that } 0 \leq u_1 \leq u_1^{max},~0 \leq u_2 \leq u_2^{max} \right\}$. Given a function $u \in L^\infty([0,T],U)$, the associated solution of the equation \eqref{eq_state} belongs to $C([0,T],C(0,1))$, which can be seen as a subset of $C([0,T],L^2(0,1))$. We define $X = L^2(0,1)$.

First, we notice that minimizing the cost $\rho_C(T)$ is equivalent to minimizing the cost $\rho_C(T) - \rho_C(0)$ (as the initial number of cells is prescribed), and it can be written under the integral form:
\begin{align*}
\rho_C(T) - \rho_C(0) 	&= \int^T_0{\rho_C'(t)\,dt} \\
				&= \int^T_0{\int^1_0{\partial_t n_C(t,x)\,dx} \,dt} \\
				&= \int^T_0{\int^1_0{\left[\frac{r_C(x)}{1 + \alpha_C u_2(t)} - d_C(x) \rho_C(t) - \mu_C(x) u_1(t) \right]n_C(t,x)\,dx}\,dt} 
\end{align*}
Thus, in view of applying the PMP, we define the function $f^0 : X \times U \rightarrow \R$ by
\begin{equation*}
f^0(n,u_1,u_2) := \int^1_0{\left[ \frac{r_C(x)}{1 + \alpha_C u_2} - d_C(x) \rho - \mu_C(x) u_1 \right]n(x) \,dx},
\end{equation*}
where $\rho := \int^1_0{n}$, and the hamiltonian is then defined by 
\begin{equation*}
H(n,u_1,u_2,p,p^0) := p^0 f^0(n,u_1,u_2) + \int^1_0{p(x)\left[ \frac{r_C(x)}{1 + \alpha_C u_2} - d_C(x) \rho - \mu_C(x) u_1 \right]n(x) \,dx}.
\end{equation*}
Since $({n}_C(\cdot),{u}(\cdot))$ is optimal, there exists a non trivial pair $(p^0, p(\cdot)) \in \R \times C([0,T],X)$, such that the adjoint equation \eqref{eq_costate} writes:
\begin{equation*}
\frac{\partial p}{\partial t}(t,x) = - \left[\frac{r_C(x)}{1 + \alpha_C u_2(t)} - d_C(x) \rho - \mu_C(x) u_1(t) \right]\cdot \left[ p(t,x) + p^0 \right] + \int^1_0{d(x) n(t,x) \left[ p(t,x) + p^0 \right]\,dx}.
\end{equation*}
Owing to Remark \ref{rq_freefinalstate}, we know that $p^0<0$.

Let us set $\tilde{p} := p + p^0$, which satisfies
\begin{equation*}
\frac{\partial \tilde{p}}{\partial t}(t,x) = - \left[\frac{r_C(x)}{1 + \alpha_C u_2(t)} - d_C(x) \rho - \mu_C(x) u_1(t) \right] \tilde{p}(t,x)  + \int^1_0{d(x) n(t,x) \tilde{p}(t,x) \,dx}.
\end{equation*}
The transversality equation \eqref{eq_transverse} yields $p(T, \cdot) = 0$, i.e., $\tilde{p}(T) = p^0$.

Then, in order to exploit the maximisation condition \eqref{eq_maximisation}, we can split the hamiltonian as
\begin{equation*}
H(t,n_C(t),u_1(t), u_2(t), p(t),p^0) = -\int^1_0{p(t,x) d_C(x) \rho(t) n_C(t,x) \,dx} - u_1(t) \phi_1(t) + \frac{\phi_2(t)}{1 + \alpha_C u_2(t) },
\end{equation*}
where the two switching functions are defined as
\begin{align*}
\phi_1(t) & := \int^1_0{\mu_C(x) n_C(t,x) \tilde{p}(t,x) \,dx}, \\
\phi_2(t) & := \int^1_0{r_C(x) n_C(t,x) \tilde{p}(t,x) \,dx}.
\end{align*}
Thus, we derive the following rule to compute the controls :
\begin{itemize}
\item If $\phi_1(t) > 0$ (resp. $\phi_2(t) > 0$), then $u_1(t) = 0$ (resp. $u_2(t) = 0$).
\item If $\phi_1(t) < 0$ (resp. $\phi_2(t) < 0$), then $u_1(t) = u_1^{max}$ (resp. $u_2(t) = u_2^{max}$).
\end{itemize}
We compute the derivative of the switching function:
\begin{align*}
\phi_1'(t) &= \int^1_0{\mu_C(x) \left( \partial_t n_C(t,x) \tilde{p}(t,x) + n_C(t,x) \partial_t \tilde{p}(t,x)  \right) \,dx} \\
		&= \left(\int^1_0{\mu_C(x) n_C(t,x) \,dx}\right) \cdot \left( \int^1_0{d_C(x) n_C(t,x) \tilde{p}(t,x) \,dx} \right).
\end{align*}
We know that $\int^1_0{\mu_C(x) n_C(t,x) \,dx} > 0$, so that the sign of $\phi_1'(t)$ is given by the sign of
\begin{equation*}
\int^1_0{d_C(x) n_C(t,x) \tilde{p}(t,x) \,dx}.
\end{equation*}
Let us set $\psi_1(t) := \int^1_0{d_C(x) n_C(t,x) \tilde{p}(t,x) \,dx}$. The same computation as before yields
\begin{equation*}
\psi_1'(t) = \left(\int^1_0{d_C(x) n_C(t,x) \,dx}\right) \psi_1(t).
\end{equation*}
Therefore, the sign of $\psi_1(t)$ is constant, given by the sign of $\psi_1(T)  = \int^1_0{d_C(x) n_C(T,x) \tilde{p}(T,x) \,dx} = \int^1_0{d_C(x) n_C(T,x) p^0 \,dx}<0$ since $p^0<0$. This implies that the function $\phi_1$ is decreasing on $[0,T]$. Since at the final time, $\phi_1(T) < 0$, we deduce the existence of a time $t_1 \in [0,T)$ such that $\phi_1(t) \geq 0$ on $[0,t_1]$, and $\phi_1(t) < 0$ on $[t_1,T]$. The same computation yields the same result for $\phi_2$, for some time $t_2 \in [0,T]$. 
\end{proof}

%%%%%%%%%%%%%%%%%%%%%%%
%%%%%%%%%%%%%%%%%%%%%%%

\section{The Continuation Procedure}

\label{Section4}

\subsection{General Principle}
We here recall the principle of direct methods and of continuations for optimization problems. Together with Theorem \ref{th_simplified}, we then derive an algorithm to solve the problem $\bf{(OCP_1)}$. 

\paragraph{On direct methods for PDEs.} \par
Let us give an informal presentation of the principle of a direct method for the resolution of the optimal control of a PDE.
Assume that we have some evolution equation written in a general form on $[0,T] \times [0,1]$ as
\[\frac{\partial n}{\partial t}(t,x) = f(t,n(t),u(t)) + An(t,x), \; n(0) = n^0,\]
where $T$ is a fixed time, $A$ is some operator on the state space, $f$ some function which might depend non-locally on $n$, $u$ a scalar control, $t \in [0,T]$, and $x \in [0,1]$ is the space or phenotype variable. The possible boundary conditions are contained in the operator $A$, which in our case will be the Neumann Laplacian. \par
Consider the optimal control problem
\[\inf_{u\in \mathcal{U}} g(n(T)),\]
where $T$ is fixed, as a function of $u \in \mathcal{U}:= \{u \in L^\infty([0,T],\R), \, 0 \leq u(t) \leq u^{max} \text{ on }[0,T]\}$. \par
\medskip
Further assume that we have discretized this PDE both in time and space through uniform meshes $0 < t_0 < t_1 < \ldots < t_{N_t} := T$, $0 =: x_0 < x_1 < \ldots <  x_{N_x} := 1$, and that we are given some discretizations of the operator $A$ (resp. the function $f$, $g$) denoted by $A_{h}$ (resp. $f_h$, $g_{h}$), where $h:= \frac{1}{N_x}$. With a Euler scheme in time, if one writes formally $n(t_i,x_j) \approx n_{i,j}$, $u(t_i) \approx u_i$ and $n_i := (n_{i,j})_{0\leq j \leq N_x}$, we are faced with the optimization problem 

\[\inf_{u_i, \, 0 \leq i \leq N_t } g_{h} \left(n_{N_t}\right),\] subject to the constraints 
\[n_{i+1,j} = n_{i,j} + h f_{h,j}(t_i,n_{i,j},u_i) + h A_{h} (n_i), \; n_{i,0} = n^0(x_i), \; 0 \leq u_i \leq u^{max}\] for all $0 \leq i \leq N_t, \; 0\leq j \leq N_x$. Note that $f_{h,j} (t_i,n_{i,j},u_i)$ stands for the function $f_h(t_i,n_{i,j},u_i)$ evaluated at $x_j$.
\medskip
\paragraph{On continuation methods for optimization problems.} \par
The optimal control problem of a PDE becomes a finite-dimensional optimization problem once approximated through a direct method, such as the one presented above. Let us denote $\mathcal{P}_1$ this problem. As already mentioned in the introduction, the numerical resolution of such a problem requires a good initial guess for the optimal solution. The idea of a continuation is to deform the problem to an easier problem $\mathcal{P}_0$ for which we either have a very good a priori knowledge of the optimal solution, or expect the problem to be solved efficiently. \par

One then progressively transforms the problem back to the original one thanks to a continuation parameter $\lambda$, thus passing through a series of optimization problems $\left(\mathcal{P}_{\lambda}\right)$. At each step of the procedure, the optimization problem $\mathcal{P}_{\lambda+d \lambda}$ is solved by taking the solution to $\mathcal{P}_{\lambda}$ as an initial guess. \par
%
%Let us emphasize that this technique is common in the case of indirect methods for optimal control problems: the optimal control problem is transformed into a shooting problem thanks to a Pontryagin Maximum Principle. This is equivalent to finding the zeros of an appropriate function, which numerically boils down to the application of a Newton algorithm. Having an initial guess is difficult and addressing this problem is usually done thanks to continuation methods. We refer to~\cite{Trelat2008} for an introduction to continuation techniques for indirect methods in numerical optimal control, and to~\cite{Cerf2012} for an example of application of such methods.

\subsection{From $(\textbf{OCP}_{\textbf{1}})$ to $(\textbf{OCP}_{\textbf{0}})$}

Let us consider $(\textbf{OCP}_{\textbf{1}})$ and formally set the following coefficients to $0$:
 \[\beta_H,\; \beta_C, \; a_{CH}, \; \theta_H, \; \theta_{HC}.\]
Note that by setting $\beta_H$ and $\beta_C$ to $0$, we also imply that the Neumann boundary conditions are no longer enforced.  \par
When doing so, the equations on $n_C$ and $n_H$ are no longer coupled since the constraints do not play any role and the interaction itself (through $a_{CH}$) is switched off. Consequently, the optimal control problem with all these coefficients set to $0$ is precisely $(\textbf{OCP}_{\textbf{0}})$. \par
 
We now define a family of optimal controls $\left(\textbf{OCP}_{\bf{\lambda}}\right)$ where $\lambda \in \mathbb{R}^4$ has each of its components between $0$ and $1$. It is a vector because several consecutive continuations will be performed (in an order to be chosen) on the different parameters. 
For $\lambda = \left(\lambda_i\right)_{1 \leq i \leq 4}$, we use the subscript $\lambda$ for the parameters associated to the optimal control problem $\left(\textbf{OCP}_{\bf{\lambda}}\right)$, and they are defined by: 
\begin{equation*}
\beta^{(\lambda)}_H := \lambda_1 \beta_H, \; \; \beta^{(\lambda)}_C := \lambda_1 \beta_C, \; \; a^{(\lambda)}_{CH} :=\lambda_2 a_{CH}, \;  \;\theta^{(\lambda)}_{CH} :=\lambda_3 \theta_{CH}, \;  \;\theta^{(\lambda)}_{H} :=\lambda_4 \theta_{H}.
\end{equation*}
In other words, $\lambda_1$, $\lambda_2$, $\lambda_3$ and $\lambda_4$ stand for the continuations on the epimutations rates, the interaction coefficient $a_{CH}$, the constraint \eqref{contHC} and the constraint \eqref{contH}, respectively.

\subsection{General Algorithm}
Let us now explain the general approach based on the previous considerations.

\paragraph{Final objective and discretization.}
Our final aim is to solve $(\textbf{OCP}_{\textbf{1}})$ numerically, with $T$ large, and a very fine discretization in time ($N_t$ is taken to be large): $T$, $N_t$ and $N_x$ are thus fixed to certain given values. To do so, we will solve successively several problems 
$\left(\textbf{OCP}_{\bf{\lambda}}\right)$ with the same discretization paremeters. Following the general method introduced about direct methods for PDEs, numerically solving an intermediate optimal control problem $\left(\textbf{OCP}_{\bf{\lambda}}\right)$ for a given $\lambda$ will mean solving the resulting optimization problem. To be more specific, we briefly explain below how the different terms are discretized. Recall that our discretization is uniform both in time $t$ and in phenotype $x$, with respectively $N_t$ and $N_x$ points.
%%%%%%%%%%%%%
\begin{itemize}
\item The non-local terms $\rho_H, \, \rho_C$ are discretized with the rectangle method :
\begin{equation*}
\rho(t_i) = \int_{0}^{1}{n(t_i,x)\, dx} \approx \frac{1}{N_x}\sum_{j=0}^{N_x - 1}{n_{i,j}}.
\end{equation*}
\item The Neumann Laplacian is discretized by its classical discrete explicit counterpart :
\begin{equation*}
\Delta n(t_i,x_j) \approx \frac{n_{i,j+1} - 2n_{i,j} + n_{i,j-1}}{(\Delta x)^2}.
\end{equation*}
We manage to take $N_t$ large enough to make sure that the CFL \[\beta_C T \frac{\left(N_x\right)^2}{N_t} < \frac{1}{2},\] is verified. Using an implicit discretization could allow us to get rid of the CFL condition but an implicit scheme happens to be more time-consuming. Therefore, we preferred using an explicit discretization, as our procedure enables us to discretize the equations finely enough to satisfy the CFL.
\item The selection term (whose sign can be both positive or negative) is discretized through an implicit-explicit scheme to ensure unconditional stability.
\end{itemize}

\noindent
\textbf{Sketch of the algorithm.} 

\medskip
\noindent
\textit{Step 1.} 
We start the continuation by solving $(\textbf{OCP}_{\textbf{0}})$. Thanks to the result \ref{th_simplified}, finding the minimizer of the end-point mapping $\left(u_1, u_2\right) \longmapsto \rho_C(T)$ is equivalent to finding the minimizer of the application $\left(t_1, t_2 \right) \longmapsto \rho_C(T)$ where $t_1$ (resp. $t_2$) are the switching times of $u_1$ (resp. $u_2$) from $0$ to $u_1^{max}$ (resp. $u_2^{max}$), as introduced in Theorem \ref{th_simplified}. \par
Numerically, we can use an arbitrarily refined discretization of $(\textbf{OCP}_{\textbf{0}})$, since the resulting optimization problem has to be made on a $\mathbb{R}^2$-valued function, which leads to a quick and efficient resolution. \par

\medskip
\noindent
\textit{Step 2.} Once $(\textbf{OCP}_{\textbf{0}})$ has been solved numerically, we get an excellent initial guess to start performing the continuation on the parameter $\lambda$. Its different components will successively be brought from $0$ to $1$, either directly or, when needed, through a proper discretization of the interval $[0,1]$. The order in which the successive coefficients are brought to their actual values is chosen so as to reduce the run-time of the algorithm. The precise order and way in which the continuation has been carried out are detailed together with the numerical results in Section \ref{Section4}.

Let us make a few remarks on possible further continuations: 
\begin{itemize}
\item
Since the goal is to take large values for $T$, one might think of performing a continuation on the final time. We again emphasize that the interest and coherence of the method requires to start with a fine discretization at Step 1, but we note that it is also possible to further refine the discretization after Step 2.
\item
Finally, it is also possible to consider the cost as introduced in Remark \ref{rq:OtherObjective}, which can be done through a continuation on the parameter $\lambda_0$.
\end{itemize}
%%%%%%%%%%%%%
% Section numerique
%%%%%%%%%%%%%
\section{Numerical Results}

\label{Section5}
Let us now apply the algorithm with AMPL and IPOPT. 

\noindent
For our numerical experiments, we will use the following values, taken from~\cite{Lorz2013}:
\begin{equation*}
r_C(x) = \frac{3}{1+x^2},~~~ r_H(x) = \frac{1.5}{1 + x^2},
\end{equation*}
%%%%%%%
\begin{equation*}
d_C(x)= \frac{1}{2}(1 - 0.3x),~~~ d_H(x) = \frac{1}{2}(1 - 0.1x),
\end{equation*}
%%%%%%%
\begin{equation*}
\begin{array}{cccc}
a_{HH} = 1, & a_{CC} = 1, & a_{HC} =  0.07, & a_{cH} = 0.01 \\
\end{array}
\end{equation*}
%%%%%%%
\begin{equation*}
\begin{array}{cc}
\alpha_H = 0.01, & \alpha_C = 1,
\end{array}
\end{equation*}
%%%%%%%
\begin{equation*}
\mu_H = \frac{0.2}{0.7^2 + x^2},~~ \mu_C = \max \left(\frac{0.9}{0.7^2 + 0.6 x^2} - 1,0\right),
\end{equation*}
%%%%%%%
\begin{equation*}
u_1^{max} = 2,~~ u^{max}_2 = 5.
\end{equation*}

One can find in~\cite{Pouchol2016} a discussion on the choice of the functions $\mu_H$ and $\mu_C$. Also, we consider the initial data:
\begin{equation*}
n_H(0,x) = K_{H,0} \exp\left(-\frac{(x-0.5)^2}{\varepsilon}\right), ~~n_C(0,x) = K_{C,0} \exp\left(-\frac{(x-0.5)^2}{\varepsilon}\right),
\end{equation*}
with $\varepsilon = 0.1$ and $K_{H,0}$ and $K_{C,0}$ are chosen such that: 
\begin{equation*}
\rho_H(0) = 2.7,~~~ \rho_C(0) = 0.5.
\end{equation*}

The rest of the parameters (namely $\beta_H$, $\beta_C$, $\theta_H$ and $\theta_{HC}$) will depend on the case we consider, and we will specify them in what follows.

\begin{remark}
Note that we have taken $u_1^{max}$ and $u_2^{max}$ to be slightly below their values chosen in~\cite{Pouchol2016} (which makes the problem harder from the applicative point of view). This is because we are here able to let $T$ take larger values, for which the final cost obtained with the optimal strategy $\rho_C(T)$ becomes too small, see below for the related numerical difficulties. \par
As for the epimutations rates, we have proceeded as follows: we have simulated the effect of constant doses and observed the long-time behavior. In the case $\beta_H = \beta_C = 0$, we know by~\cite{Pouchol2016} that both cell densities must converge to Dirac masses. With mutations, we expect some Gaussian-like approximation of these Diracs, the variance of which was our criterion to select a suitable epimutation rate in terms of modeling. It must be large enough to observe a real variability due to the epimutations, but small enough to avoid seeing no selection effects (diffusion dominates and the steady-state looks almost constant).
\end{remark}

\paragraph{Test case 1: $T = 60$.} We recall that this case corresponds to the example presented in~\cite{Pouchol2016}, to which we add a diffusion term. We set the parameters for the diffusion to $\beta_H = 0.001$ and $\beta_C = 0.0001$. The coefficients for the constraints are $\theta_{HC} = 0.4$ and $\theta_H = 0.6$. For such numerical values, the optimal cost satisfies $\rho_C(T) << 1$, which can be source of numerical difficulties. To overcome this, we introduce the following trick: let us define $u^{max,0}_1 = 1$ and $u^{max,0}_2 = 4$. We apply the procedure described in Section \ref{Section3} with the values $u^{max,0}_1$ and $u^{max,0}_2$. We then add another continuation step by raising them to the original desired values $u_1^{max} = 2$ and $u_2^{max} = 5$. In the formalism previously introduced, it amounts to adding two continuation parameters $\lambda_5$ and $\lambda_6$ to the vector $\lambda = (\lambda_i)_{1 \leq i \leq 4}$. The parameters associated to the optimal control problem $(\textbf{OCP}_{\bf{\lambda}})$ are then defined as :
\[
u_1^{max,(\lambda)} := (1 - \lambda_5) u_1^{max,0} + \lambda_5 u_1^{max},~~ u_2^{max,(\lambda)} := (1 - \lambda_6) u_2^{max,0} + \lambda_6 u_2^{max}.
\]
More precisely, we perform the continuation in the following way, summarized in Figure \ref{fig_procedure_1}:
\begin{itemize}
\item First, we solve $(\textbf{OCP}_{\textbf{0}})$, with $u^{max,0}_1 = 1$ and $u^{max,0}_2 = 4$.
\item Second, we add the interaction between the two populations, the diffusion parameters, and the constraint on the number of healthy cells. That is, the parameters $a_{CH}$, $\beta_H$, $\beta_C$ and $\theta_H$ are set to their values.
\item Then, we add the constraint measuring the ratio between the number of healthy cells and the total number of cells, that is $\theta_{HC}$.
\item Lastly, we raise the maximum values for the controls from $u^{max,0}_i$ to $u^{max}_i$ ($i \in \left\{1,2 \right\}$), and we solve $(\textbf{OCP}_{\textbf{1}})$ for $T = 60$.
\end{itemize}

\begin{figure}[h!]
\begin{center}
\includegraphics[scale=0.9]{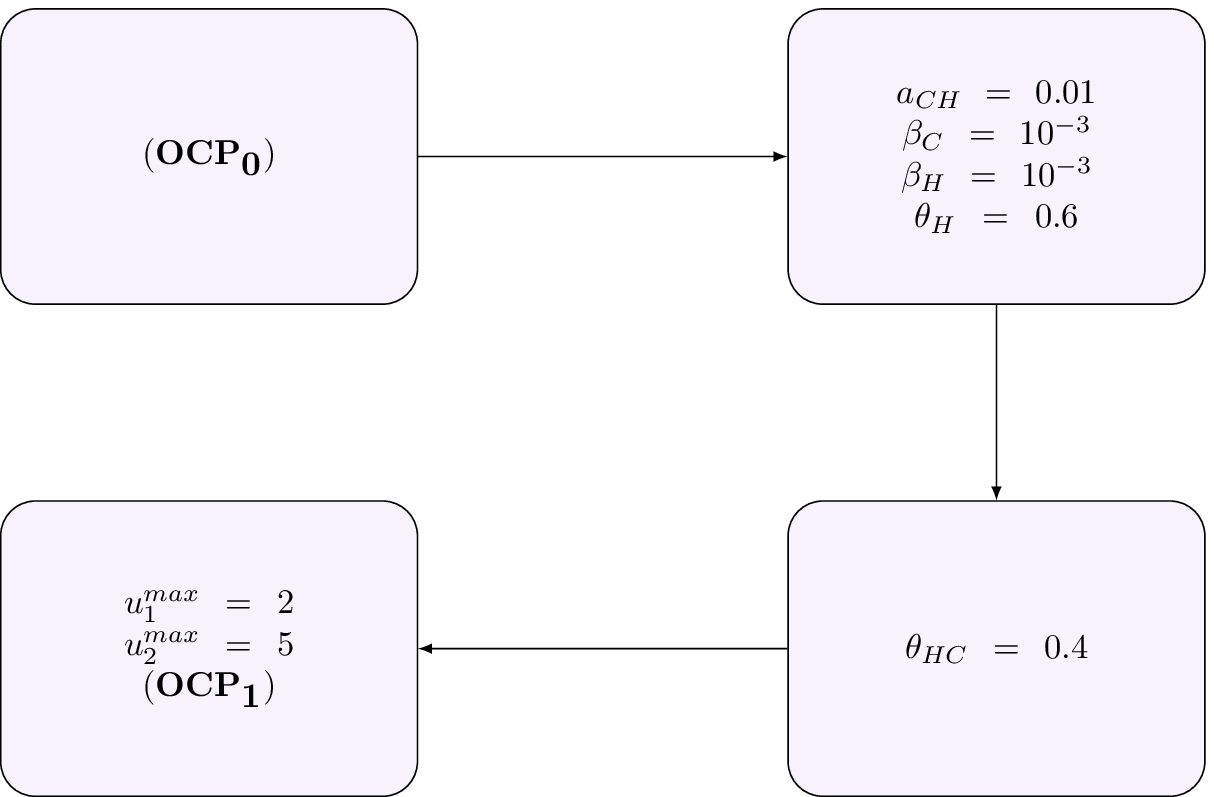}
\end{center}
\caption{Continuation procedure to solve $(\textbf{OCP}_{\textbf{1}})$ for $T = 60$.}
\label{fig_procedure_1}
\end{figure}

Actually, for this set of parameters, only four consecutive resolutions are required to solve $(\textbf{OCP}_{\textbf{1}})$ starting from $(\textbf{OCP}_{\textbf{0}})$. That is, the components of the continuation vector $\lambda = (\lambda_i)_{1\leq i \leq 6}$ are brought directly from $0$ to $1$, taking no intermediate value, in the order schematized on Figure \ref{fig_procedure_1}. We will study further in the paper a case for a larger final time, for which having a more refined discretization is mandatory.

On Figure \ref{fig_trace1}, we plot the optimal controls $u_1$ and $u_2$ at the four steps of the continuation procedure. We also display the evolution of the constraint on the size of the tumor compared to the healthy tissue \eqref{contHC}. We can clearly identify the emergence of the expected structure for the controls, namely a long phase along which the constraint \eqref{contHC} saturates, followed by a bang arc with $u_1 = u_1^{max}$ and $u_2 = u_2^{max}$, and a last boundary arc along which the constraint \eqref{contH} saturates. Throughout this section, we will use a red solid line in our figures for $(\textbf{OCP}_{\textbf{1}})$, a green solid line for $(\textbf{OCP}_{\textbf{0}})$ and a dotted style for anything refering to $(\textbf{OCP}_{\bf{\lambda}})$.

\begin{figure}[!h]
\hspace{-0.8cm}
\includegraphics[scale=0.6]{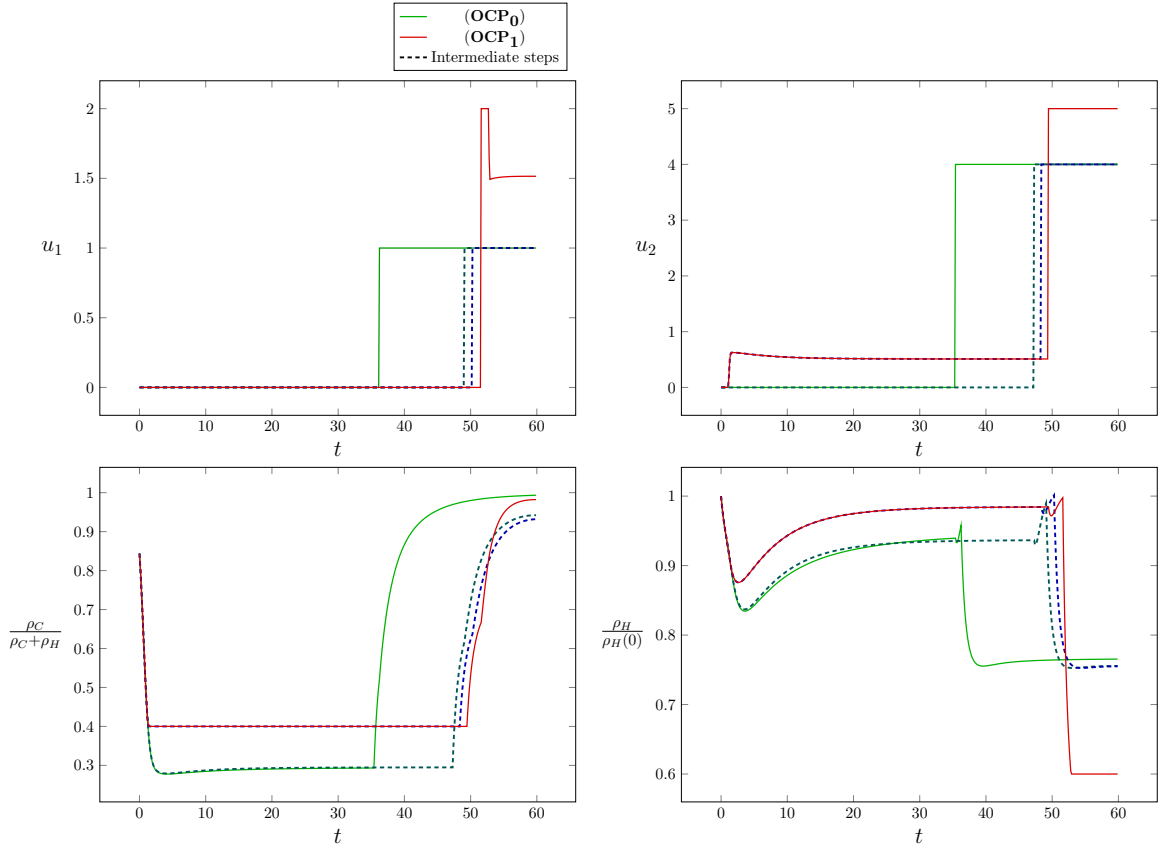}
\caption{Intermediate steps of the continuation procedure for the test case 1.}
\label{fig_trace1}
\end{figure}

\begin{remark}
We would like to emphasize here that our procedure enables us to use a much more refined discretization of the problem than what was done in~\cite{Pouchol2016}. More precisely, we discretize with $N_t = 500$ and $N_x = 20$ points in our direct method. For such a discretization, directly tackling $(\textbf{OCP}_{\textbf{1}})$ with the direct method fails.
\end{remark}

\begin{remark}
Note that the constraint $\rho_H / \rho_H(0) > 0.6$ does not saturate until the last step of the continuation, when raising the maximal value of the controls. Therefore, when we add it at the beginning of the procedure, it is not actually active.
\end{remark}

\paragraph{Test case 2: $T = 80$.}
Whereas one could believe that raising the final time from $T = 60$ to $T = 80$ does not much increase the difficulty of the problem, we noticed that several numerical obstacles appeared. In the following, we consider a discretization with $N_t = 250$ and $N_x = 12$ points, in order to keep the optimization run-time reasonable. Besides, in order to test the robustness of our procedure, we consider a more restrictive constraint on the number of healthy cells: we choose $\theta_H = 0.75$ ($0.6$ in the first example).

First, we use the same numerical trick as explained in our first example, reducing the maximal value for the controls to $u_1^{max,0} = 0.7$ and $u_2^{max,0} = 3.5$. For given values of $u_1^{max}$ and $u_2^{max}$, the optimal cost $\rho_C(T)$ decreases when $T$ increases. This is why we now use smaller values of $u_1^{max,0}$ and $u_2^{max,0}$, compared to the first example where we set them to respectively $1$ and $4$.

We performed the continuation in the following way, summarized in Figure \ref{fig_procedure_2}:
\begin{itemize}
\item First, we solve $(\textbf{OCP}_{\textbf{0}})$, with $u^{max,0}_1 = 0.7$ and $u^{max,0}_2 = 3.5$.
\item Second, we add the interaction between the two populations (via the parameter $a_{CH}$), and the constraint measuring the ratio between the number of healthy cells and the total number of cells \eqref{contHC} is introduced at the intermediate value $\theta^{(\lambda)}_{HC} = 0.3$.
\item We then raise it to its final value of $\theta_{HC} = 0.4$.
\item As a fourth step, we simultaneously add the constraint \eqref{contH} on the healthy cells and raise the maximal values for the controls from $u^{max,0}_i$ to $u^{max}_i$ ($i \in \left\{1,2 \right\}$).
\item Lastly, we add diffusion to the model, via the parameters $\beta_H$ and $\beta_C$, and we solve $(\textbf{OCP}_{\textbf{1}})$ for $T = 80$.
\end{itemize}

\begin{figure}
\begin{center}
\includegraphics[scale = 0.9]{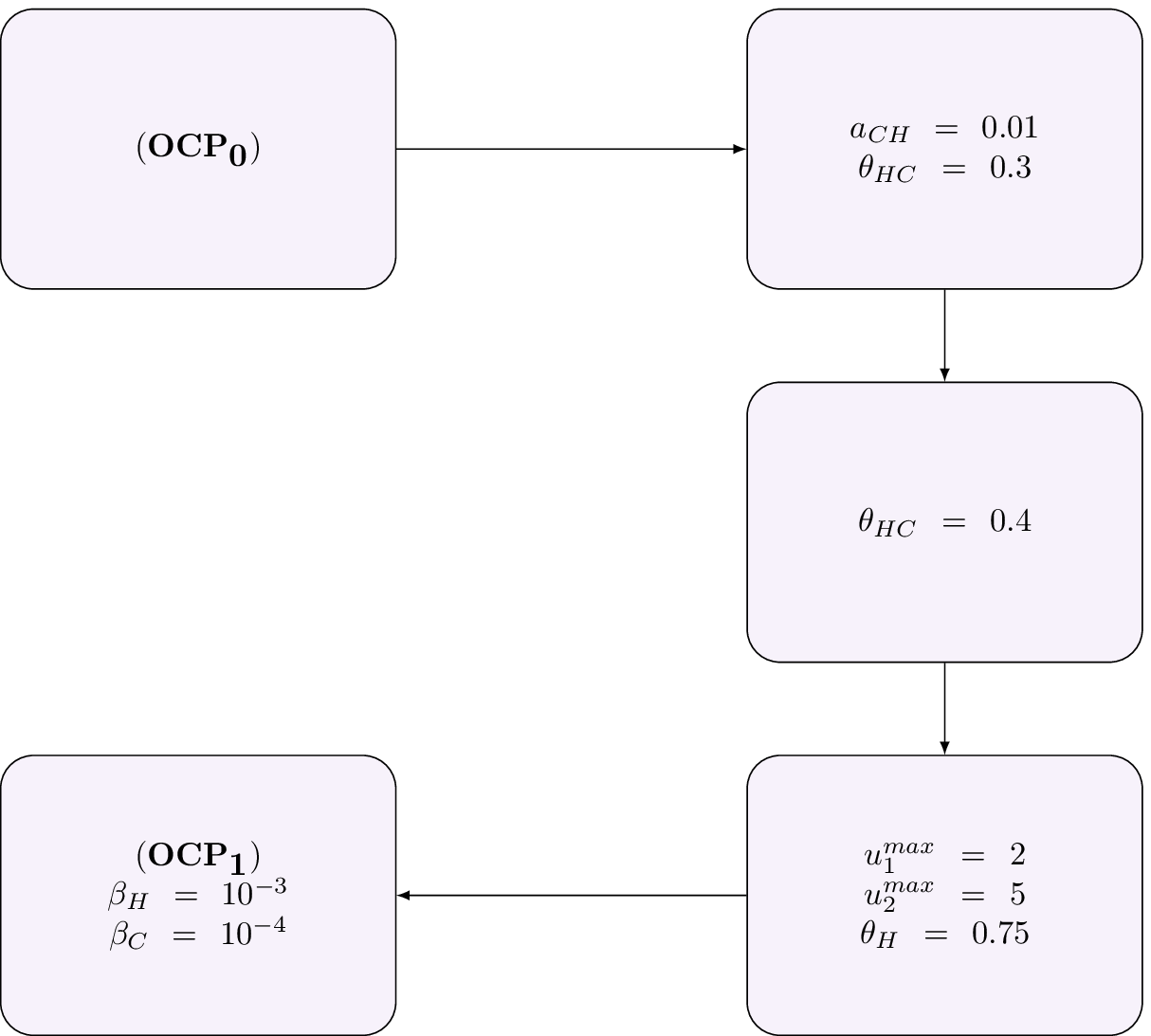}
\end{center}
\caption{Continuation procedure to solve $(\textbf{OCP}_{\textbf{1}})$ for $T = 80$.}
\label{fig_procedure_2}
\end{figure}

At this point, we need to make two important remarks concerning this continuation procedure.
\begin{remark}
The order in which we make the components of the continuation vector $\lambda = (\lambda_i)_{1 \leq i \leq 6}$ vary from $0$ to $1$ is different from the order we presented for $T = 60$. For instance, we noticed that the diffusion makes the problem significantly harder to solve, although the Laplacians where discretized using the simplest explicit finite-difference approximation. Therefore, we only added it at the last step of the continuation.
\end{remark}

\begin{remark}
Whereas for $T = 60$, raising the $(\lambda_i)_{1 \leq i \leq 6}$ directly from $0$ to $1$ was enough to solve $(\textbf{OCP}_{\textbf{1}})$, it became necessary to use a more refined discretization for $T=80$. This fact justifies the principle of our continuation procedure, as each step is necessary to solve the next one, and thus $(\textbf{OCP}_{\textbf{1}})$ in the end.
 For instance, on Figure \ref{fig_contHC}, we display the evolution of the constraint \eqref{contHC}:
\begin{equation*}
\frac{\rho_H(t)} {\rho_C(t) + \rho_H(t)} \geq \lambda_3 \theta_{HC}
\end{equation*}
when raising the continuation parameter $\lambda_3$ from $0$ to $1$.
On Figure \ref{fig_umax}, we display the evolution of the controls $u_1$ and $u_2$ when raising their maximal allowed values from $(u_1^{max,0},u_2^{max,0})$ to $(u_1^{max},u_2^{max})$. For the sake of readability, we do not show all the steps of the continuation, but only some of them. It clearly shows how the structure of the optimal solution evolves from the simple one of $(\textbf{OCP}_{\textbf{0}})$ to the much more complex one of $(\textbf{OCP}_{\textbf{1}})$.
\end{remark}

\begin{figure}
\begin{center}
\hspace{-1.3cm}
\includegraphics[scale=0.95]{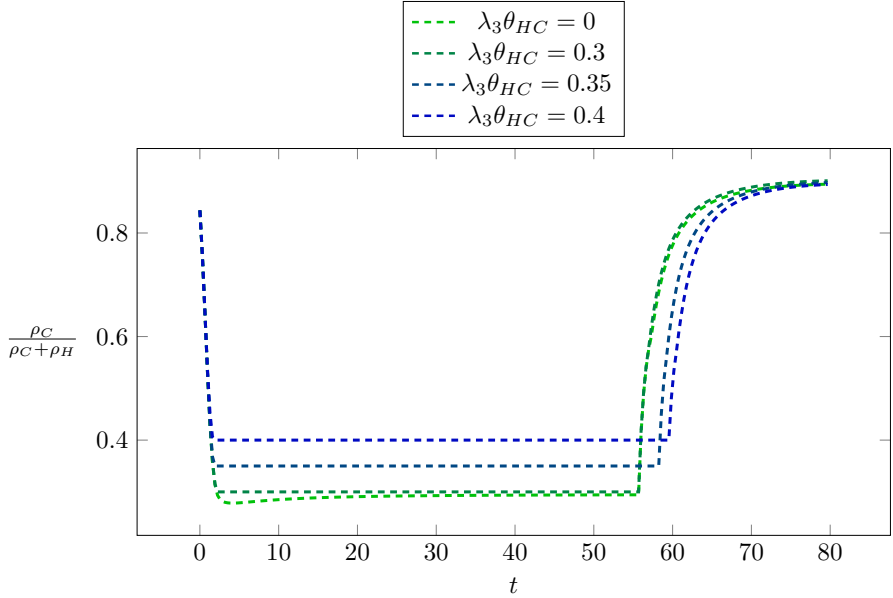}
\end{center}
\caption{Evolution of the constraint \eqref{contHC} during the continuation.}
\label{fig_contHC}
\end{figure}

\begin{figure}
\hspace{-0.8cm}
\includegraphics[scale=0.62]{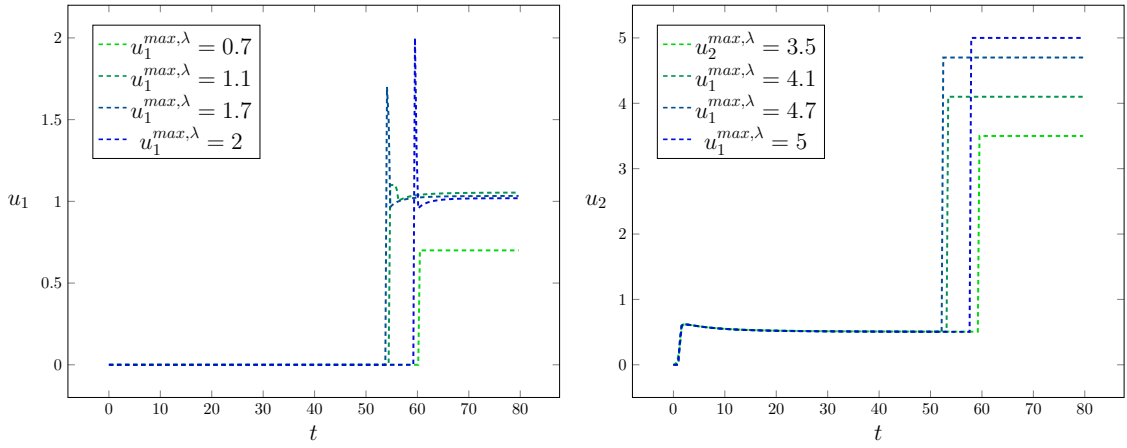}

\caption{Raising the maximal values $u_1^{max}$, $u_2^{max}$ for the controls.}
\label{fig_umax}
\end{figure}

Finally, we display on Figure \ref{fig_nC} the evolution of $n_C$, when applying the optimal strategy we found solving $(\textbf{OCP}_{\textbf{1}})$. In black we represent the initial condition $n_C(0,\cdot)$, and with lighter shades of red, the evolution of $n_C(t,x)$ as time increases.
\begin{figure}
\begin{center}
\hspace{-0.7cm}
\includegraphics[scale=1.1]{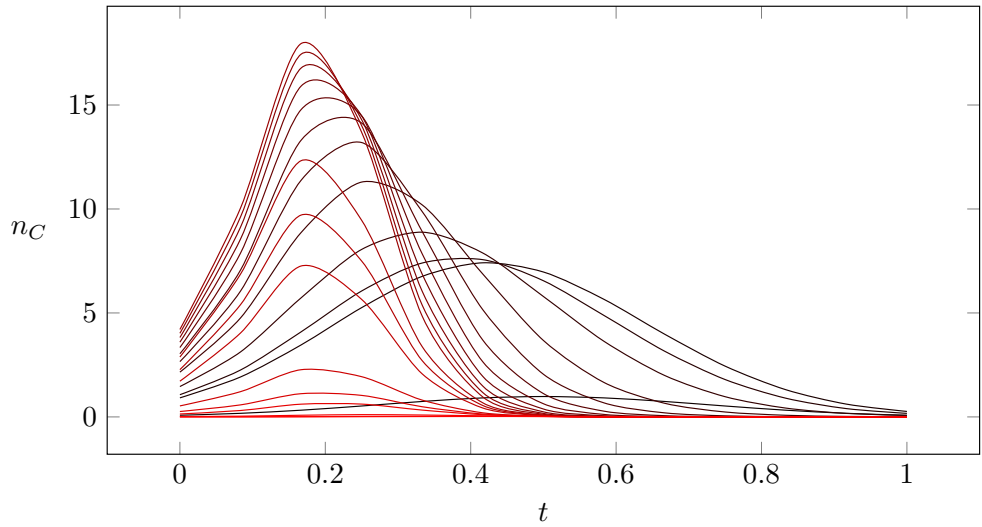}
\end{center}
\vspace{-0.2cm}
\caption{Evolution of $n_C$ for the optimal solution of $(\textbf{OCP}_{\textbf{1}})$.}
\label{fig_nC}
\end{figure}

One clearly sees that the optimal strategy has remained the same: the cancer cell population concentrates on a sensitive phenotype, which is the key idea to then use the maximal tolerated doses. In other words, the strategy identified in the previous work~\cite{Pouchol2016} is robust with respect to addition of epimutations. An important remark is that the cost obtained with the optimal strategy is higher with the mutations than without them: this is because we cannot have convergence to a Dirac located at a sensitive phenotype, but to a smoothed (Gaussian-like) version of that Dirac. There will always be residual resistant cells which will make the second phase less successful.

\paragraph{Further comments on the continuation principle.} A continuation procedure can be used in a wide range of applications, and one can easily imagine ways to generalize the ideas we have previously introduced. Let us illustrate our point with an example: we have presented a procedure to solve $(\textbf{OCP}_{\textbf{1}})$, for some initial conditions $n_H^0$ and $n_C^0$. Suppose that we wish to solve $(\textbf{OCP}_{\textbf{1}})$ for some different initial conditions $\tilde{n}_H^0$ and $\tilde{n}_C^0$. Biologically, this could correspond to finding a control strategy for a different tumor. A natural idea is then to use a continuation procedure to deform the problem from the initial conditions $(n_H^0,n_C^0)$ to $(\tilde{n}_H^0,\tilde{n}_C^0)$, rather than applying again the whole procedure to solve $(\textbf{OCP}_{\textbf{1}})$ with $\tilde{n}_H^0$ and $\tilde{n}_C^0$. We successfully performed some numerical tests to validate this idea: if we dispose of a set of initial conditions for which we want to solve $(\textbf{OCP}_{\textbf{1}})$, it is indeed faster to solve $(\textbf{OCP}_{\textbf{1}})$ for one of them and then perform a continuation on the initial data, rather than solving $(\textbf{OCP}_{\textbf{1}})$ for each of the initial conditions. More generally, any parameter in the model could lend itself to a continuation.

\paragraph{Test case 3: $T = 60$, more general objective function.} 
The optimal strategy obtained with the previous objective function $\rho_C(T)$ might seem surprising, in particular because it advocates for very limited action at the beginning: giving no cytotoxic drugs and low loses of cytostatic drugs. To further investigate the robustness of this strategy, let us also consider the objective function $\lambda_0 \int_0^T \rho_C(s) \,ds + (1 - \lambda_0)\rho_C(T)$ as introduced in Remark \ref{rq:OtherObjective}, for different values of $\lambda_0$. To ease numerical computations, we take $\beta_H = \beta_C = 0$, $u_1^{max}= 2$, $u_2^{max} = 5$, and finally $N_x = 20$, $N_t = 100$. The results are reported on Figure \ref{fig:coutL1}.

\begin{figure}
\begin{center}
\includegraphics[scale=0.6]{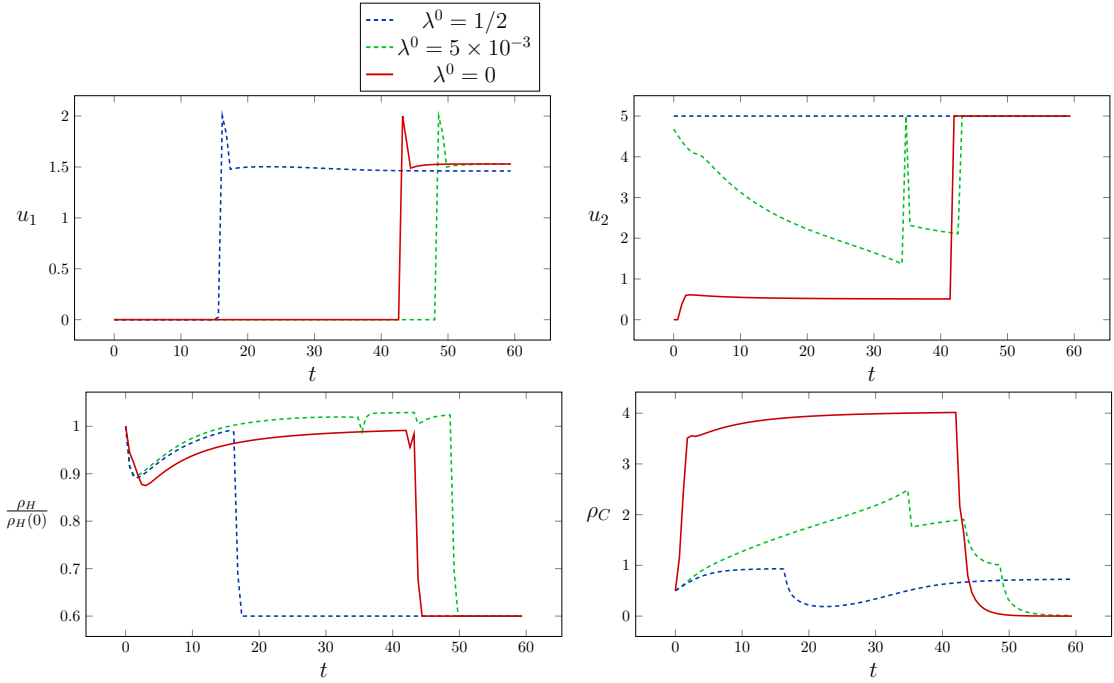}
\end{center}
\vspace{-0.2cm}
\caption{Adding a term accounting for the $L^1$ norm $\int \rho_C$ in the cost.}
\label{fig:coutL1}
\end{figure}

For $\lambda_0 = 0.5$, the $L^1$ term is dominant in the optimization and the variations of $\rho_C$ are smaller over the interval $]0,T[$. However, although there is a significant change in the control $u_2$ which is always equal to $u_2^{max}$, $u_1$ has kept the same structure: an arc with no drugs, a short arc with maximal doses and a final arc with intermediate doses. The only (though important) difference is that the first arc is not a long one as before.  \par

We infer from these numerical simulations that the optimal structure is inherent in the equations: there is no choice but to let the cancer cell density concentrate on a sensitive phenotype. Since at $\lambda_0 = 0.5$, the integral term dominates, we also consider other convex combinations with smaller values of $\lambda_0$ up to $0$, for which $u_2$ takes intermediate values before being equal to $u_2^{max}$, while $u_1 = 0$ on a longer arc. 

Among these families of objectives (depending on $\lambda_0$) and their outcomes, it is up to the oncologist to decide which one is best, depending also on what is not modeled here, for example metastases. 
%%%%%%%%%%%%%
% Section perspectives
%%%%%%%%%%%%%
\section{Perspectives}

\label{Section6}

For epimutations with rates in reasonable ranges, we found that the optimal strategy obtained in~\cite{Pouchol2016} is preserved, which is a proof of its robustness. We believe that robustness can further be tested for more complicated models, with the same strategy.

For example, one may want to model longer-range mutations by a non-local alternative to the Laplacian, either through a mutation term through a Kernel~\cite{Bonnefon2015}, or through a non-local operator like a fractional Laplacian~\cite{Cabre2013}. These could both be added by continuation, on the Kernel starting from the integro-differential model, or on the fractional exponent for the fractional Laplacian, starting from the case of the (classical) Laplacian.

Another (local) possibility is to choose a more general elliptic operator. In particular, one can think of putting a drift term to model the \textit{stress-induced adaptation}~\cite{Coville2013a, Chisholm2016a}, namely epimutations that occur because cells actively change their phenotype in a certain direction depending on the environment created by the drug.

Finally, other objective functions can also be considered through a continuation as already introduced in the present article: one minimizes a convex combination of $\rho_C(T)$ and the objective function of interest.

We refer to~\cite{Pouchol2016} for other possible generalizations of the model that might be of interest.

%%%%%%%%%%%%%%
% Section conclusion
%%%%%%%%%%%%%%
\section{Conclusions}

\label{Section7}

The objective of the present work was to numerically solve an optimal control generalizing the one studied in the article~\cite{Pouchol2016}, in which epimutations were neglected. We have developed an approach which significantly reduces the computation time and improves precision, even without mutations. More precisely, by setting enough parameters to $0$ in the original optimal control problem, we arrive to a situation where the problem can be tackled by a Pontryagin Maximum Principle in infinite dimension. Direct methods and continuation then allow to solve the problem of interest, with the strong improvement that we actually start the continuation with a very refined discretization.

We advocate that this approach is suitable for many complicated optimal controls problems. This would be the case as soon as an appropriate simplification leads to a problem for which precise results can be obtained by a PMP. In particular, this approach is an option to be investigated for optimal control problems which have a high-dimensional discretized counterpart.

\bibliographystyle{unsrt}

\bibliography{CitationNumericalOCP.bib}

\end{document}